\documentclass[11pt,a4paper]{article}

\usepackage[T1]{fontenc}
\usepackage[utf8]{inputenc}
\usepackage{lmodern}
\usepackage{amsmath,amssymb,amsthm,mathtools}
\usepackage[a4paper,left=27mm,right=27mm,top=24mm,bottom=25mm]{geometry}
\usepackage{microtype}
\usepackage{tikz}
\usepackage[hidelinks,hypertexnames=false]{hyperref}

\hypersetup{
  pdftitle={A Counterexample to the Closed Nodal-Line Conjecture for the Third Eigenfunction of the Planar Dirichlet Laplacian},
  pdfauthor={Zikang Deng},
  pdfsubject={Dirichlet Laplacian and closed nodal lines of a third eigenfunction},
  pdfkeywords={Dirichlet Laplacian, third eigenfunction, nodal line, thin-neck domain, mixed boundary conditions, spectral convergence}
}

\allowdisplaybreaks

\theoremstyle{definition}
\newtheorem{theorem}{Theorem}[section]
\newtheorem{lemma}[theorem]{Lemma}
\newtheorem{proposition}[theorem]{Proposition}
\newtheorem{corollary}[theorem]{Corollary}
\newtheorem{conjecture}[theorem]{Conjecture}
\theoremstyle{remark}
\newtheorem{remark}[theorem]{Remark}

\begin{document}

\begin{center}
{\bfseries\fontsize{16}{22}\selectfont
A Counterexample to the Closed Nodal-Line Conjecture for the Third
Eigenfunction of the Planar Dirichlet Laplacian\par}
\vspace{12pt}
{\fontsize{11}{14.4}\selectfont Zikang Deng\par}
{\fontsize{11}{14.4}\selectfont Beijing Normal University\par}
\vspace{28.4pt}
{\bfseries Abstract\par}
\end{center}
\vspace{-3pt}

In Conjecture 4.3 of volume 6 (2003) of the \emph{LMS Journal of Computation and Mathematics}, Levitin and Yagudin conjectured that if the third Dirichlet eigenvalue of a connected planar domain is simple, then it is impossible for all the nodal lines of a corresponding eigenfunction to be closed. We construct a counterexample to this conjecture. We prove that there exists a bounded connected planar domain $\Omega$ with $C^\infty$ boundary such that $\lambda_3(\Omega)$ is a simple eigenvalue and the entire interior nodal set of any corresponding nonzero real eigenfunction is compactly contained in $\Omega$; moreover, this nodal set consists of exactly two disjoint real-analytic simple closed curves.

The construction starts from the counterexample to Payne's conjecture given by D\"ahne, G\'omez-Serrano, and Hou: its second eigenvalue is simple, and its second eigenfunction is positive at an interior point and strictly negative on a Jordan curve surrounding that point. We take two mirror-image copies of this base domain and connect them by a symmetric thin neck. Reflection decomposes the full-domain spectrum into the Dirichlet--Neumann mixed spectrum and the full Dirichlet spectrum on a half-domain. Using a transverse Poincar\'e inequality and the min--max principle, we prove step by step that both spectral sequences converge to the spectrum of the base domain as the thin neck vanishes. Boundary unique continuation then shows that the mixed eigenvalue is strictly below the full Dirichlet eigenvalue of the same index, thereby placing the even second branch exactly at the third spectral position of the full domain. Local convergence of eigenfunctions preserves two sets of strict sign inequalities, while Courant's nodal domain theorem confines the entire nodal set to two fixed compact subdomains. Finally, a smooth inner exhaustion preserves the spectral gaps and the strict sign inequalities. The construction therefore disproves the Levitin--Yagudin conjecture stated above.



\section{Introduction}

Let $\Omega\subset\mathbb R^2$ be a bounded connected domain. We study the Dirichlet eigenvalue problem for a fixed membrane,
\begin{equation}
\begin{cases}
-\Delta u=\lambda u, & x\in\Omega,\\
u=0, & x\in\partial\Omega.
\end{cases}
\tag{1.1}
\end{equation}
The eigenvalues, listed with multiplicity, are
\begin{equation}
0<\lambda_1(\Omega)<\lambda_2(\Omega)\leq\lambda_3(\Omega)\leq\cdots.
\tag{1.2}
\end{equation}
Denote the corresponding real eigenfunctions by $u_1,u_2,\ldots$. For a nonzero eigenfunction $u$, let
\begin{equation}
\mathcal Z(u)=\{x\in\Omega:u(x)=0\}
\tag{1.3}
\end{equation}
be its interior nodal set; the connected components of $\Omega\setminus\mathcal Z(u)$ are called the nodal domains of $u$. Courant's nodal domain theorem asserts that an eigenfunction associated with $\lambda_k$ has at most $k$ nodal domains~\cite{CourantHilbert1953}.

The nodal set identifies the positions at which a vibrating membrane remains at rest in a given mode. Even for only the first few spectral positions, whether a nodal line meets the boundary is influenced simultaneously by the topology and symmetry of the domain and by eigenvalue multiplicity. Payne conjectured in 1967 that the nodal line of a second Dirichlet eigenfunction on a bounded planar domain must meet the boundary~\cite{Payne1967}. This statement holds for convex planar domains~\cite{Alessandrini1994,Melas1992}, but Hoffmann-Ostenhof, Hoffmann-Ostenhof, and Nadirashvili constructed a multiply connected counterexample~\cite{HoffmannOstenhofEtAl1997}. D\"ahne, G\'omez-Serrano, and Hou later obtained a counterexample with only six holes by means of rigorous computer-assisted estimates~\cite{DahneGomezSerranoHou2021}; Freitas and Leylekian subsequently obtained the corresponding phenomenon on a doubly connected domain~\cite{FreitasLeylekian2025}. These works show that, although a closed nodal line of a second eigenfunction is subject to strong restrictions, it can indeed occur on a general multiply connected domain.

In studying the range of the first three eigenvalues of the planar Dirichlet Laplacian, Levitin and Yagudin investigated $\lambda_2/\lambda_1$ and $\lambda_3/\lambda_1$ by combining shape perturbations with numerical experiments. They formulated the following conjecture in~\cite[Conjecture~4.3]{LevitinYagudin2003}.

\begin{conjecture}[Levitin--Yagudin]
Let $\Omega$ be a connected planar domain. If $\lambda_3(\Omega)$ is a simple Dirichlet eigenvalue, then it is impossible for all the nodal lines of a corresponding eigenfunction to be closed.
\end{conjecture}

Conjecture~1.1 is not equivalent to Payne's conjecture for the second eigenfunction. Courant's theorem only bounds the number of nodal domains of a third eigenfunction by three; it does not rule out the possibility that two interior closed nodal lines divide the domain into three nodal domains. On the other hand, directly juxtaposing two Payne counterexamples produces a disconnected domain and usually introduces spectral multiplicity; after the two parts are connected by a thin neck, both the ordering of the low-lying eigenvalues and the location of the nodal set may change. A counterexample must therefore solve three problems simultaneously: identify the third spectral position after the domains are connected, prove that this eigenvalue is simple, and prevent the corresponding nodal lines from entering the thin neck or meeting the outer boundary.

Our main result is the following.

\begin{theorem}
There exists a bounded connected planar domain $\Omega\subset\mathbb R^2$ with $C^\infty$ boundary such that
\begin{equation}
\lambda_2(\Omega)<\lambda_3(\Omega)<\lambda_4(\Omega).
\tag{1.4}
\end{equation}
If $U$ is any nonzero real eigenfunction associated with $\lambda_3(\Omega)$, then
\begin{equation}
\mathcal Z(U)\Subset\Omega.
\tag{1.5}
\end{equation}
Moreover, $\mathcal Z(U)$ consists of exactly two disjoint real-analytic simple closed curves.
\end{theorem}

\begin{corollary}
The Levitin--Yagudin Conjecture~1.1 is false.
\end{corollary}

The proof starts from a base domain $A$ whose second eigenvalue is simple and whose second eigenfunction satisfies a strict interior sign condition. We take two mirror-image copies of $A$ and connect them by a symmetric thin neck of width $2\varepsilon$; denote the resulting domain by $\Omega_\varepsilon$. If $H_\varepsilon$ is the right half-domain and $S_\varepsilon$ is the cross-section in the thin neck, then even eigenfunctions on the full domain induce a Neumann condition on $S_\varepsilon$, whereas odd eigenfunctions induce a Dirichlet condition. Denote the two sequences of half-domain eigenvalues by $\nu_j(\varepsilon)$ and $\delta_j(\varepsilon)$. We shall prove that
\begin{equation}
\nu_j(\varepsilon),\ \delta_j(\varepsilon)\longrightarrow a_j
\qquad(\varepsilon\downarrow0),
\tag{1.6}
\end{equation}
where $a_j$ is the $j$th Dirichlet eigenvalue of $A$. At the same time, for every fixed $\varepsilon>0$ and $j\geq1$, we have the strict inequality
\begin{equation}
\nu_j(\varepsilon)<\delta_j(\varepsilon).
\tag{1.7}
\end{equation}
Since $a_1<a_2<a_3$, it follows that
\begin{equation}
\nu_1<\delta_1<\nu_2<\min\{\delta_2,\nu_3\},
\tag{1.8}
\end{equation}
and hence the third eigenvalue of the full domain is precisely the simple even branch $\nu_2$.

The second eigenfunction of the base domain is positive at a point $p$ and strictly negative on a Jordan curve $\Gamma$ surrounding $p$. Local uniform convergence preserves these inequalities in the two mirror-image copies of the base domain, thereby producing two distinct positive nodal domains compactly contained in the full domain; the negative values on the curves also produce a negative nodal domain. Courant's theorem bounds the total number of nodal domains by three. Consequently, any zero lying outside the two Jordan domains would produce a new positive nodal domain in its neighborhood, a contradiction. This first shows that the entire nodal set is compactly contained in the domain. The finite analytic graph structure of an interior nodal set in two dimensions, together with Euler's formula, then shows that the nodal set consists of exactly two regular simple closed curves. Finally, we fix one such polygonal domain and approximate it by a connected $C^\infty$ inner exhaustion; the spectral gaps and strict sign inequalities persist for all sufficiently large exhaustion indices.

The organization of the paper follows the compact arrangement of the original paper in which the problem was posed. Section~2 records the variational framework and the sign-changing property of interior zeros, and extracts the required base domain from a known Payne counterexample. Section~3 constructs the symmetric thin-neck domain and proves the even--odd spectral decomposition. Section~4 gives a complete proof of spectral and eigenfunction convergence in the thin-neck limit. Section~5 proves strict separation of the mixed spectra and identifies the third spectral position. Section~6 controls the nodal domains and the topology of the nodal set. Section~7 completes the proof of Theorem~1.2 by a smooth inner exhaustion. We emphasize that the statement disproved here is Conjecture~4.3 of~\cite{LevitinYagudin2003}; the present counterexample does not pass judgment on the other assertions in that paper concerning domains that locally maximize spectral ratios.

\section{Preliminaries and the base domain}

\subsection{Variational formulation and nodal domains}

For a bounded connected Lipschitz domain $G\subset\mathbb R^2$, define on $H_0^1(G)$
\begin{equation}
q_G[w]=\int_G|\nabla w|^2\,dx,
\qquad
R_G[w]=\frac{q_G[w]}{\displaystyle\int_G|w|^2\,dx}
\quad(w\neq0).
\tag{2.1}
\end{equation}
The Dirichlet eigenvalues, counted with multiplicity, satisfy the min--max formula
\begin{equation}
\lambda_j(G)
=\min_{\substack{E\subset H_0^1(G)\\ \dim E=j}}
\ \max_{0\neq w\in E}R_G[w].
\tag{2.2}
\end{equation}
Since $G$ is connected, $\lambda_1(G)$ is simple, and the first eigenfunction may be chosen to be strictly positive.

The following elementary property will be used later to exclude zeros outside the two Jordan domains.

\begin{lemma}[An interior zero must be sign-changing]
Let $G\subset\mathbb R^2$ be a domain, and suppose that $u\not\equiv0$ satisfies $-\Delta u=\lambda u$ in $G$, where $\lambda>0$. If $x_0\in G$ and $u(x_0)=0$, then every neighborhood of $x_0$ contains both a point at which $u>0$ and a point at which $u<0$.
\end{lemma}

\begin{proof}
Choose $r>0$ such that $B_r(x_0)\Subset G$. If $u\geq0$ in one such ball, then
\[
-\Delta u=\lambda u\geq0,
\]
and $u$ attains its minimum value $0$ at the interior point $x_0$. The strong minimum principle implies that $u$ vanishes identically in $B_r(x_0)$; unique continuation for elliptic equations then implies that $u\equiv0$ in $G$, a contradiction. The same argument applied to $-u$ rules out the possibility that $u\leq0$ in some ball. Both signs therefore occur in every neighborhood.
\end{proof}

We shall also use the standard local structure of nodal sets in two dimensions. If $u$ satisfies $-\Delta u=\lambda u$ in a planar domain, then $u$ is real analytic in the interior. Near a regular zero, its nodal set is a real-analytic arc. If $x_0$ is a singular zero of vanishing order $m\geq2$, then, in suitable local coordinates, the lowest-order term is a nonzero homogeneous harmonic polynomial of degree $m$. Consequently, exactly $2m$ nodal arcs meet at $x_0$, and the signs in adjacent sectors alternate. This result and its finite-graph consequence for compact nodal sets may be found in Cheng~\cite{Cheng1976}. In Section~6 we shall give the graph-theoretic counting needed here in full.

\subsection{A base domain with a strict interior sign condition}

We use only a previously proved result from the literature as an input; all subsequent thin-neck and spectral-ordering arguments are given in this paper.

\begin{proposition}[Base domain]
There exists a bounded connected polygonal domain $A\subset\mathbb R^2$ whose Dirichlet spectrum
\begin{equation}
0<a_1<a_2<a_3\leq\cdots
\tag{2.3}
\end{equation}
is simple at the second spectral position. There exist a real eigenfunction $\phi$ associated with $a_2$, a polygonal Jordan domain $Q\Subset A$, a point $p\in Q$, and $\Gamma=\partial Q$ such that
\begin{equation}
\phi(p)>0,
\qquad
\max_{x\in\Gamma}\phi(x)<0.
\tag{2.4}
\end{equation}
\end{proposition}

\begin{proof}
Take the regular hexagonal domain with six holes constructed by D\"ahne, G\'omez-Serrano, and Hou in~\cite{DahneGomezSerranoHou2021}. They first prove that there are at most four eigenvalues below $67.23$, with multiplicities included in the count. The three pairwise disjoint certified regions obtained by the method of particular solutions satisfy the following: the first spectral interval contains at least one eigenvalue;
\begin{equation}
I_2=[63.21-6.89\times10^{-3},\,63.21+6.89\times10^{-3}]
\tag{2.5}
\end{equation}
contains at least one eigenvalue; and the subsequent certified cluster contains at least two eigenvalues counted with multiplicity. All four certified eigenvalues lie below $67.23$. Hence the sum of the foregoing ``at least'' counts is already four, while there are at most four eigenvalues below the threshold. Thus $I_2$ contains exactly one eigenvalue counted with multiplicity. This eigenvalue is $a_2$, so $a_2$ is simple and $a_1<a_2<a_3$.

Section~6 of that paper constructs a polygonal Jordan curve $\Gamma$ in the domain and takes $p=(1/10,0)$. For the correspondingly normalized and oriented second eigenfunction, their certified estimates give
\begin{equation}
\max_{\Gamma}\phi\leq-2.767\times10^{-6},
\qquad
\phi(p)\geq0.013374568.
\tag{2.6}
\end{equation}
More specifically, the upper bound for the approximate eigenfunction on the curve is $-4.4929\times10^{-5}$, and the global pointwise error does not exceed $4.2162\times10^{-5}$; at $p$, the approximate value lies in $[0.01342\pm3.27\times10^{-6}]$. Combining the corresponding terms gives (2.6), and hence (2.4).
\end{proof}

Hereafter we fix $A,Q,\Gamma,p,\phi$ from Proposition~2.2 and write the strict margin as
\begin{equation}
\eta=\min\left\{\phi(p),-\max_{\Gamma}\phi\right\}>0.
\tag{2.7}
\end{equation}

\section{Symmetric thin-neck domains and the even--odd decomposition}

Choose a straight segment of the outer boundary of $A$ that is disjoint from $\overline Q$, and take a point $b$ in its relative interior. After a rigid motion, we may assume that this segment is vertical at the point of attachment. Fix $L>0$. Place a copy $A_+$ of $A$ to the right of the line $\{x_1=L\}$ so that the attachment window is
\begin{equation}
I_\varepsilon=\{L\}\times(-\varepsilon,\varepsilon),
\tag{3.1}
\end{equation}
and let $A_-$ be the reflection of $A_+$ across $\{x_1=0\}$. Take the rectangular thin neck
\begin{equation}
N_\varepsilon=(-L,L)\times(-\varepsilon,\varepsilon)
\tag{3.2}
\end{equation}
and glue its two end windows to the corresponding boundary windows of $A_-$ and $A_+$, respectively. More precisely, define
\begin{equation}
\Omega_\varepsilon
=\operatorname{int}\bigl(\overline{A_-}\cup\overline{N_\varepsilon}\cup\overline{A_+}\bigr).
\tag{3.3}
\end{equation}
When $\varepsilon$ is less than half the length of the selected straight segment, this is a bounded, connected, Lipschitz polygonal domain symmetric about $x_1=0$.

Let
\begin{equation}
H_\varepsilon=\Omega_\varepsilon\cap\{x_1>0\},
\qquad
S_\varepsilon=\{0\}\times(-\varepsilon,\varepsilon).
\tag{3.4}
\end{equation}
The half-domain consists of $A_+$ and the half-neck $T_\varepsilon=(0,L)\times(-\varepsilon,\varepsilon)$. We always impose the Dirichlet condition on the physical boundary $\partial H_\varepsilon\setminus S_\varepsilon$, and impose either the natural Neumann condition or the Dirichlet condition on $S_\varepsilon$. The corresponding form domains are
\begin{align}
V_\varepsilon^N
&=\{w\in H^1(H_\varepsilon):\operatorname{Tr}w=0
\text{ on }\partial H_\varepsilon\setminus S_\varepsilon\},
\tag{3.5}\\
V_\varepsilon^D
&=H_0^1(H_\varepsilon).
\tag{3.6}
\end{align}
Both forms are given by $q_\varepsilon[w]=\int_{H_\varepsilon}|\nabla w|^2$. Denote the two eigenvalue sequences, counted with multiplicity, by
\begin{equation}
0<\nu_1(\varepsilon)\leq\nu_2(\varepsilon)\leq\cdots,
\qquad
0<\delta_1(\varepsilon)\leq\delta_2(\varepsilon)\leq\cdots.
\tag{3.7}
\end{equation}

\begin{lemma}[Even--odd spectral decomposition]
The Dirichlet spectrum of the full domain $\Omega_\varepsilon$, counted with multiplicity, is the disjoint labeled union of the two spectra:
\begin{equation}
\operatorname{spec}(-\Delta_{\Omega_\varepsilon}^D)
=\{\nu_j(\varepsilon):j\geq1\}\mathbin{\dot\cup}
\{\delta_j(\varepsilon):j\geq1\}.
\tag{3.8}
\end{equation}
The eigenvalues $\nu_j$ correspond to eigenfunctions that are even with respect to $x_1=0$, and the eigenvalues $\delta_j$ correspond to odd eigenfunctions.
\end{lemma}

\begin{figure}[htbp]
\centering
\begin{tikzpicture}[x=0.72cm,y=0.72cm,line join=round,line cap=round]
  \draw
    (-3.9,0.55)--(-3.9,1.40)--(-4.4,2.25)--(-7.1,2.10)
    --(-7.5,-0.10)--(-7,-1.60)--(-4.4,-1.60)--(-3.9,-0.55)
    --(3.9,-0.55)--(4.4,-1.60)--(7,-1.60)--(7.5,-0.10)
    --(7.1,2.10)--(4.4,2.25)--(3.9,1.40)--(3.9,0.55)--cycle;
  \draw[densely dashed] (0,0.55)--(0,3.0) node[above] {$x_1=0$};
  \draw[line width=1pt] (0,0.55)--(0,-0.55);
  \node[above right,fill=white,inner sep=1pt] at (0,0.55) {$S_\varepsilon$};
  \draw (-5.9,1.05) ellipse (0.75 and 0.45);
  \draw (5.9,1.05) ellipse (0.75 and 0.45);
  \fill (-5.9,1.05) circle (1.8pt);
  \fill (5.9,1.05) circle (1.8pt);
  \node[anchor=east] at (-6.02,1.05) {$p_-$};
  \node[anchor=west] at (6.02,1.05) {$p_+$};
  \node at (-5.9,1.72) {$Q_-$};
  \node at (5.9,1.72) {$Q_+$};
  \node at (-5.9,0.43) {$\Gamma_-$};
  \node at (5.9,0.43) {$\Gamma_+$};
  \foreach \x/\y in {-6.8/0.20,-5.0/0.20,-6.8/-0.75,-5.9/-0.62,-5.0/-0.75,6.8/0.20,5.0/0.20,6.8/-0.75,5.9/-0.62,5.0/-0.75}
    \draw (\x,\y) circle (0.11);
  \node at (-5.9,-1.20) {$A_-$};
  \node at (5.9,-1.20) {$A_+$};
  \node at (0,-0.88) {thin neck $N_\varepsilon$};
\end{tikzpicture}
\caption{Schematic of the symmetric thin-neck domain; the small circles only indicate the holes in the base domain, and the figure is not drawn to scale.}
\end{figure}
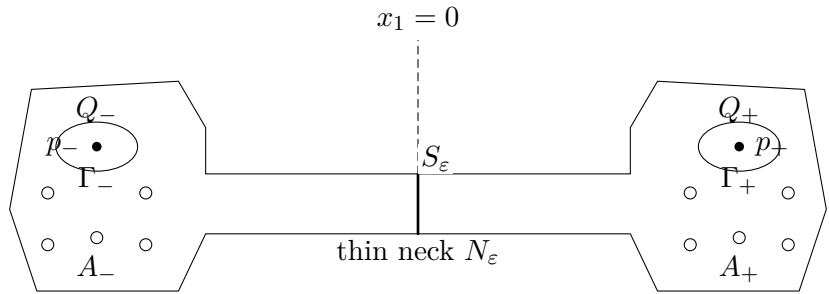

\begin{proof}
Define $\mathcal Rf(x_1,x_2)=f(-x_1,x_2)$. The reflection $\mathcal R$ is a self-adjoint unitary operator on $L^2(\Omega_\varepsilon)$, preserves $H_0^1(\Omega_\varepsilon)$, and satisfies
\[
q_{\Omega_\varepsilon}[\mathcal Rf,\mathcal Rg]
=q_{\Omega_\varepsilon}[f,g].
\]
Hence the Dirichlet Laplacian commutes with $\mathcal R$, and
\[
L^2(\Omega_\varepsilon)=L^2_{\mathrm{even}}\oplus L^2_{\mathrm{odd}}
\]
is an orthogonal decomposition into two reducing subspaces.

After restriction to $H_\varepsilon$, an even function belongs to $V_\varepsilon^N$; on the cross-section $S_\varepsilon$, its weak formulation gives the natural condition $\partial_{x_1}u=0$. Conversely, the even extension of a function in $V_\varepsilon^N$ belongs to $H_0^1(\Omega_\varepsilon)$. Under restriction and even extension, both the $L^2$ norm squared and the energy change by the same factor of $2$, so the Rayleigh quotient is unchanged. Similarly, an odd function has zero trace on $S_\varepsilon$, and its restriction belongs to $V_\varepsilon^D$; the odd extension of a function in $V_\varepsilon^D$ gives an odd function on the full domain. Thus the spectra of the two reduced operators are $\{\nu_j\}$ and $\{\delta_j\}$, respectively, and merging them with multiplicity gives (3.8).
\end{proof}

\section{Spectral convergence in the thin-neck limit}

Rather than reducing spectral convergence to a one-line invocation of a varying-domain result, we give the complete variational proof needed for the present construction. General theories of spectral convergence on varying domains and domains with thin tubes may be found in~\cite{Daners2003,CollinsTaylor2018,FelliOgnibene2020,RauchTaylor1975}.

\begin{lemma}[Neck mass estimate]
For every $w\in V_\varepsilon^N$ (and hence also for every $w\in V_\varepsilon^D$),
\begin{equation}
 \int_{T_\varepsilon}|w|^2\,dx
 \leq \frac{4\varepsilon^2}{\pi^2}
       \int_{T_\varepsilon}|\partial_{x_2}w|^2\,dx
 \leq \frac{4\varepsilon^2}{\pi^2}q_\varepsilon[w].
 \tag{4.1}
\end{equation}
\end{lemma}

\begin{proof}
The two long sides $x_2=\pm\varepsilon$ of the thin neck belong to the physical boundary. Hence, for almost every $x_1\in(0,L)$, the slice
\[
 x_2\longmapsto w(x_1,x_2)
\]
belongs to $H_0^1(-\varepsilon,\varepsilon)$. The first Dirichlet eigenvalue of the interval is $\pi^2/(4\varepsilon^2)$, and the one-dimensional Poincar\'e inequality gives
\[
 \int_{-\varepsilon}^{\varepsilon}|w(x_1,x_2)|^2\,dx_2
 \leq \frac{4\varepsilon^2}{\pi^2}
       \int_{-\varepsilon}^{\varepsilon}|\partial_{x_2}w(x_1,x_2)|^2\,dx_2.
\]
Integrating with respect to $x_1$ yields (4.1). Whether a Neumann or Dirichlet condition is imposed on the cross-section $S_\varepsilon$ has no effect on this transverse estimate.
\end{proof}

\begin{proposition}[Thin-neck limit of the two spectral sequences]
For each fixed $j\geq1$, as $\varepsilon\downarrow0$,
\begin{equation}
 \nu_j(\varepsilon)\longrightarrow a_j,
 \qquad
 \delta_j(\varepsilon)\longrightarrow a_j.
 \tag{4.2}
\end{equation}
\end{proposition}

\begin{proof}
The proof is identical for the two conditions on the cross-section. In what follows, $e_j(\varepsilon)$ denotes either $\nu_j(\varepsilon)$ or $\delta_j(\varepsilon)$, and the corresponding form domain is denoted by $V_\varepsilon$.

We first prove the upper-limit estimate. Let $\phi_1,\ldots,\phi_j$ be the first $j$ $L^2(A)$-orthonormal Dirichlet eigenfunctions of $A$. Extend each $\phi_k$ by zero into $T_\varepsilon$. Since $\phi_k\in H_0^1(A)$, its trace on the connecting window $I_\varepsilon$ vanishes; the gluing lemma for Sobolev functions therefore shows that the zero extension belongs to $H^1(H_\varepsilon)$. It vanishes on the entire physical boundary and on $S_\varepsilon$, and hence belongs simultaneously to $V_\varepsilon^N$ and $V_\varepsilon^D$. This extension preserves both the $L^2$ inner products and the Dirichlet energies. Substituting $\operatorname{span}\{\phi_1,\ldots,\phi_j\}$ into the min--max formula gives
\begin{equation}
 e_j(\varepsilon)\leq a_j,
 \qquad
 \limsup_{\varepsilon\downarrow0}e_j(\varepsilon)\leq a_j.
 \tag{4.3}
\end{equation}

We next prove the lower-limit estimate. Let $\varepsilon_n$ be any sequence converging to zero, and pass to a subsequence along which $e_j(\varepsilon_n)$ converges to its lower limit. For $1\leq k\leq j$, choose the first $j$ $L^2(H_{\varepsilon_n})$-orthonormal eigenfunctions $w_{k,n}$. By (4.3),
\begin{equation}
 \int_{H_{\varepsilon_n}}|\nabla w_{k,n}|^2
 =e_k(\varepsilon_n)\leq a_j.
 \tag{4.4}
\end{equation}
Lemma~4.1 therefore gives
\begin{equation}
 \int_{T_{\varepsilon_n}}|w_{k,n}|^2
 \leq \frac{4a_j}{\pi^2}\varepsilon_n^2
 \longrightarrow0.
 \tag{4.5}
\end{equation}

Restrict $w_{k,n}$ to the fixed copy $A_+$ and identify the latter, by a rigid motion, with a function on $A$. By (4.4) and the normalization, these restrictions are bounded in $H^1(A)$. After a diagonal extraction, there exist $\psi_k\in H^1(A)$ such that
\begin{equation}
 w_{k,n}|_A\rightharpoonup\psi_k \quad\text{in }H^1(A),
 \qquad
 w_{k,n}|_A\longrightarrow\psi_k \quad\text{in }L^2(A).
 \tag{4.6}
\end{equation}
The second convergence follows from the Rellich compact embedding.

It remains to verify that $\psi_k$ has a full Dirichlet trace. Let $g_{k,n}$ be the trace of $w_{k,n}|_A$ on $\partial A$. Continuity of the trace operator implies that $g_{k,n}$ is bounded in $L^2(\partial A)$, and it vanishes on $\partial A\setminus I_{\varepsilon_n}$. For every fixed $h\in L^2(\partial A)$,
\begin{equation}
 \left|\int_{\partial A}g_{k,n}h\,ds\right|
 \leq C\,\|h\mathbf{1}_{I_{\varepsilon_n}}\|_{L^2(\partial A)}
 \longrightarrow0,
 \tag{4.7}
\end{equation}
because the length of the window tends to zero and the $L^2$ integral is absolutely continuous. On the other hand, (4.6) and weak continuity of the trace operator show that $g_{k,n}$ converges weakly to $\operatorname{Tr}\psi_k$. Thus $\operatorname{Tr}\psi_k=0$, that is,
\begin{equation}
 \psi_k\in H_0^1(A).
 \tag{4.8}
\end{equation}
Observe that this argument uses only the Lipschitz trace theorem and the fact that the window length tends to zero; the fixed holes in the base domain create no additional difficulty.

It follows from (4.5), orthonormality on the full domain, and (4.6) that
\begin{equation}
 \int_A\psi_k\psi_\ell\,dx=\delta_{k\ell}.
 \tag{4.9}
\end{equation}
In particular, $E=\operatorname{span}\{\psi_1,\ldots,\psi_j\}$ has dimension $j$. Now take any $c=(c_1,\ldots,c_j)\in\mathbb{R}^j$ and set
\[
 W_n=\sum_{k=1}^j c_k w_{k,n},
 \qquad
 \Psi=\sum_{k=1}^j c_k\psi_k.
\]
By weak lower semicontinuity, orthogonality of the eigenfunctions, and $e_k(\varepsilon_n)\leq e_j(\varepsilon_n)$,
\begin{align}
 \int_A|\nabla\Psi|^2
 &\leq\liminf_{n\to\infty}\int_{H_{\varepsilon_n}}|\nabla W_n|^2 \notag\\
 &=\liminf_{n\to\infty}\sum_{k=1}^j c_k^2e_k(\varepsilon_n) \notag\\
 &\leq\left(\liminf_{n\to\infty}e_j(\varepsilon_n)\right)
       \sum_{k=1}^j c_k^2.
 \tag{4.10}
\end{align}
Equation (4.9) gives $\|\Psi\|_{L^2(A)}^2=\sum_k c_k^2$. Taking the maximum over the $j$-dimensional space $E$ and then applying the min--max formula on $A$, we obtain
\begin{equation}
 a_j\leq\max_{0\neq\Psi\in E}R_A[\Psi]
 \leq\liminf_{n\to\infty}e_j(\varepsilon_n).
 \tag{4.11}
\end{equation}
Since the original sequence was arbitrary, (4.3) and (4.11) together prove (4.2).
\end{proof}

\begin{proposition}[Interior convergence of the simple branch]
There exists $\varepsilon_1>0$ such that $\nu_2(\varepsilon)$ is simple whenever $0<\varepsilon<\varepsilon_1$. Let $v_\varepsilon$ be a real eigenfunction corresponding to this eigenvalue, normalized by $\|v_\varepsilon\|_{L^2(H_\varepsilon)}=1$. Fix its sign by requiring
\begin{equation}
 \int_A v_\varepsilon\phi\,dx>0.
 \tag{4.12}
\end{equation}
Then
\begin{equation}
 v_\varepsilon|_A\longrightarrow\phi
 \quad\text{in }L^2(A),
 \tag{4.13}
\end{equation}
and, for every $K\Subset A$,
\begin{equation}
 v_\varepsilon\longrightarrow\phi
 \quad\text{in }C^1(K).
 \tag{4.14}
\end{equation}
\end{proposition}

\begin{proof}
Proposition~4.2 with $j=1,2,3$ gives $\nu_j(\varepsilon)\to a_j$. Since $a_1<a_2<a_3$, after decreasing $\varepsilon_1$ we have
\[
 \nu_1(\varepsilon)<\nu_2(\varepsilon)<\nu_3(\varepsilon),
\]
and hence $\nu_2(\varepsilon)$ is simple. Temporarily leave the sign unspecified, and choose any normalized eigenfunction $\widetilde v_\varepsilon$. By the compactness argument used in the preceding proof, every sequence converging to zero has a subsequence along which $\widetilde v_\varepsilon|_A$ converges strongly in $L^2(A)$ to some $\psi\in H_0^1(A)$; moreover, the vanishing of the neck mass gives $\|\psi\|_2=1$. For any $\zeta\in C_c^\infty(A)$, when $\varepsilon$ is sufficiently small, $\zeta$ may be used as a test function in the half-domain eigenvalue equation, so that
\[
 \int_A\nabla\widetilde v_\varepsilon\cdot\nabla\zeta
 =\nu_2(\varepsilon)\int_A\widetilde v_\varepsilon\zeta.
\]
Passing to the limit gives $-\Delta\psi=a_2\psi$. Since $a_2$ is simple, $\psi=\phi$ or $\psi=-\phi$. Therefore,
\[
 \left|\int_A\widetilde v_\varepsilon\phi\,dx\right|\longrightarrow1.
\]
After decreasing $\varepsilon_1$ once more, this inner product is nonzero, and the sign can therefore be fixed uniquely by (4.12). For the thus oriented $v_\varepsilon$, every subsequential limit above must equal $\phi$; consequently the whole family satisfies (4.13).

Choose nested compact sets $K\Subset K_1\Subset K_2\Subset A$, and let $z_\varepsilon=v_\varepsilon-\phi$. On $K_2$, $z_\varepsilon$ satisfies
\[
 -\Delta z_\varepsilon
 =\nu_2(\varepsilon)z_\varepsilon
  +\bigl(\nu_2(\varepsilon)-a_2\bigr)\phi.
\]
Denote the right-hand side by $f_\varepsilon$. By (4.13) and convergence of the eigenvalues, $f_\varepsilon\to0$ in $L^2(K_2)$. The second-order interior estimate yields
\[
 \|z_\varepsilon\|_{H^2(K_1)}
 \leq C\bigl(\|f_\varepsilon\|_{L^2(K_2)}
             +\|z_\varepsilon\|_{L^2(K_2)}\bigr)
 \longrightarrow0.
\]
It follows that $f_\varepsilon\to0$ in $H^1(K_1)$. Applying a higher-order interior estimate gives
\[
 \|z_\varepsilon\|_{H^3(K)}
 \leq C\bigl(\|f_\varepsilon\|_{H^1(K_1)}
             +\|z_\varepsilon\|_{L^2(K_1)}\bigr)
 \longrightarrow0.
\]
The two-dimensional embedding $H^3(K)\hookrightarrow C^1(K)$ now gives (4.14).
\end{proof}

\section{Strict separation of the mixed spectra and identification of the third eigenvalue}

The inclusion $V_\varepsilon^D\subset V_\varepsilon^N$ yields only a non-strict inequality. The following argument uses unique continuation across the boundary to rule out equality.

\begin{theorem}[Strict separation at equal indices]
For every $\varepsilon>0$ and every $j\geq1$,
\begin{equation}
 \nu_j(\varepsilon)<\delta_j(\varepsilon).
 \tag{5.1}
\end{equation}
\end{theorem}

\begin{proof}
The inclusion $V_\varepsilon^D\subset V_\varepsilon^N$ and the min--max principle first give $\nu_j\leq\delta_j$. Suppose, to the contrary, that for some $j$,
\begin{equation}
 \nu_j=\delta_j=\alpha.
 \tag{5.2}
\end{equation}
Let $E_D$ be the space spanned by the first $j$ orthonormal eigenfunctions of the full Dirichlet problem on the half-domain. Then $\dim E_D=j$, and
\begin{equation}
 q_\varepsilon[f]\leq\alpha\|f\|_2^2
 \qquad(f\in E_D).
 \tag{5.3}
\end{equation}
Let $F_N^{<\alpha}$ be the spectral subspace spanned by all eigenfunctions of the mixed Neumann problem whose eigenvalues are strictly below $\alpha$. Since $\nu_j=\alpha$, we have $\dim F_N^{<\alpha}\leq j-1$. Hence one may choose
\begin{equation}
 0\neq f\in E_D\cap(F_N^{<\alpha})^\perp.
 \tag{5.4}
\end{equation}
Expand $f$ in the orthonormal eigenbasis of the mixed Neumann problem. The orthogonality condition shows that only terms with eigenvalues at least $\alpha$ occur in this expansion. Therefore,
\begin{equation}
 q_\varepsilon[f]\geq\alpha\|f\|_2^2.
 \tag{5.5}
\end{equation}
Together with (5.3), this shows that equality holds on both sides. Equality in the Neumann spectral expansion means that $f$ belongs to its $\alpha$-eigenspace; equality in the Dirichlet spectral expansion on $E_D$ means that $f$ also belongs to the $\alpha$-eigenspace of the full Dirichlet problem.

Choose any open subsegment $S'$ of the cross-section $S_\varepsilon$ whose closure stays away from both endpoints. The full Dirichlet condition gives $f|_{S'}=0$, while the natural boundary condition for the mixed problem gives $\partial_n f|_{S'}=0$. The boundary is a straight line near the cross-section, and elliptic regularity gives these two Cauchy data their classical meaning. Extend $f$ by zero across $S'$. Since both the function value and the normal flux vanish, integration by parts produces no distributional interface term, so the extended function still satisfies $(-\Delta-\alpha)f=0$ in the distributional sense. It vanishes identically on an open set on the other side of the interface; interior unique continuation forces it to vanish identically in the connected half-domain, contradicting (5.4). Thus equality is impossible, and (5.1) follows.
\end{proof}

\begin{proposition}[Location and simplicity of the third eigenvalue]
There exists $\varepsilon_0>0$ such that, whenever $0<\varepsilon<\varepsilon_0$,
\begin{equation}
 \nu_1(\varepsilon)<\delta_1(\varepsilon)<\nu_2(\varepsilon)
 <\min\{\delta_2(\varepsilon),\nu_3(\varepsilon)\}.
 \tag{5.6}
\end{equation}
Consequently,
\begin{equation}
 \lambda_1(\Omega_\varepsilon)=\nu_1(\varepsilon),
 \qquad
 \lambda_2(\Omega_\varepsilon)=\delta_1(\varepsilon),
 \qquad
 \lambda_3(\Omega_\varepsilon)=\nu_2(\varepsilon),
 \tag{5.7}
\end{equation}
and $\lambda_3(\Omega_\varepsilon)$ is simple.
\end{proposition}

\begin{proof}
By (2.3), one may choose $\rho>0$ such that
\begin{equation}
 2\rho<\min\{a_2-a_1,a_3-a_2\}.
 \tag{5.8}
\end{equation}
Proposition~4.2 shows that, when $\varepsilon$ is sufficiently small, both $\nu_k$ and $\delta_k$ lie within distance $\rho$ of $a_k$ for $k=1,2,3$. Hence $\delta_1<\nu_2$ and $\nu_2<\nu_3$. Theorem~5.1 also gives $\nu_1<\delta_1$ and $\nu_2<\delta_2$; together these inequalities yield (5.6).

By Lemma~3.1, the spectrum on the full domain is the union, with multiplicities, of the two spectral sequences. Equation (5.6) shows that the first three terms of the merged sequence are, in order, $\nu_1,\delta_1,\nu_2$, which proves (5.7). Moreover, $\nu_1<\nu_2<\nu_3$ excludes multiplicity of $\nu_2$ within the even subspace, while $\delta_1<\nu_2<\delta_2$ excludes equality with any odd eigenvalue. Thus $\lambda_3=\nu_2$ is simple on the full space, and
\[
 \lambda_2<\lambda_3<\lambda_4.
\]
\end{proof}

\section{Nodal-Set Control and a Polygonal Counterexample}

Fix a sufficiently small \(\varepsilon>0\) satisfying Proposition~5.2. Let
\(v_\varepsilon\) be the \(\nu_2\)-eigenfunction oriented as in (4.12), and
let \(U_\varepsilon\) be its even extension across \(S_\varepsilon\). By
Lemma~3.1 and Proposition~5.2, \(U_\varepsilon\) is an eigenfunction
corresponding to the simple third eigenvalue.

Denote by \(Q_\pm,\Gamma_\pm,p_\pm\) the images of \(Q,\Gamma,p\),
respectively, in the two reflected copies. These compact sets remain at a
positive distance from the connecting windows. Proposition~4.3 and the
margin in (2.7) show that, after decreasing \(\varepsilon\) further,
\begin{equation}
  U_\varepsilon(p_\pm)>\frac{\eta}{2},
  \qquad
  \max_{x\in\Gamma_\pm}U_\varepsilon(x)<-\frac{\eta}{2}.
  \tag{6.1}
\end{equation}

\begin{theorem}[Polygonal counterexample]\label{thm:polygonal-counterexample}
For a fixed sufficiently small \(\varepsilon>0\), the domain
\(\Omega_\varepsilon\) satisfies \(\lambda_2<\lambda_3<\lambda_4\), and the
interior nodal set of its third eigenfunction \(U_\varepsilon\) satisfies
\begin{equation}
  \mathcal Z(U_\varepsilon)\Subset Q_+\cup Q_-\Subset\Omega_\varepsilon.
  \tag{6.2}
\end{equation}
More precisely, \(\mathcal Z(U_\varepsilon)\) consists of exactly two
disjoint real-analytic simple closed curves, one in each of \(Q_+\) and
\(Q_-\).
\end{theorem}

\begin{proof}
We first control the nodal domains. Let \(P_\pm\) be the connected component
of the positive set \(\{U_\varepsilon>0\}\) that contains \(p_\pm\). By
(6.1), \(U_\varepsilon\) is strictly negative on the Jordan curve
\(\Gamma_\pm=\partial Q_\pm\). Every continuous path from \(p_\pm\) to the
exterior of \(Q_\pm\) must cross \(\Gamma_\pm\). Hence the positive component
cannot cross this curve, and therefore \(P_\pm\subset Q_\pm\). Continuity and
the strict negative margin on the curve also yield a neighborhood of
\(\Gamma_\pm\) on which the function is negative; consequently,
\begin{equation}
  \overline{P_\pm}\Subset Q_\pm.
  \tag{6.3}
\end{equation}
The two Jordan domains are disjoint, so \(P_+\) and \(P_-\) are two distinct
positive nodal domains. On the other hand, (6.1) directly shows that the
negative set is nonempty, and hence there is at least one negative nodal
domain. Thus \(U_\varepsilon\) has at least three nodal domains. Since it
occupies the third spectral position on the full domain, Courant's theorem
gives at most three nodal domains. The nodal domains are therefore precisely
\(P_+\), \(P_-\), and a unique negative nodal domain \(D_-\). In particular,
\begin{equation}
  \{U_\varepsilon>0\}=P_+\cup P_-.
  \tag{6.4}
\end{equation}

Take any \(x\in\mathcal Z(U_\varepsilon)\). By Lemma~2.1, every neighborhood
of \(x\) contains a positive point, and hence
\[
  x\in\overline{\{U_\varepsilon>0\}}
   =\overline{P_+}\cup\overline{P_-}.
\]
Combining this with (6.3) gives (6.2). This already proves that none of the
interior nodal lines meets \(\partial\Omega_\varepsilon\).

It remains to prove the stronger assertion that there are two simple closed
curves. Write \(Z=\mathcal Z(U_\varepsilon)\). Because \(Z\) is compactly
contained in the domain, the local structure of planar nodal sets
\cite{Cheng1976} makes \(Z\) a finite embedded planar graph: a regular point
has local degree \(2\), every singular vertex has degree \(2m\geq4\), and
there are no endpoints of degree \(1\). On each closed-curve component
without singular vertices, insert an artificial vertex of degree \(2\).
Let \(V,E,c,F\) denote, respectively, the numbers of vertices, edges,
connected components of the graph, and connected components of
\(\mathbb R^2\setminus Z\).

Each \(Q_\pm\) contains a zero: the endpoints of a curve joining \(p_\pm\) to
a point of \(\Gamma_\pm\) have opposite signs, so the intermediate value
theorem gives a zero. Since \(Q_+\) and \(Q_-\) are separated, \(c\geq2\).
By the first part of the proof, \(P_+\) and \(P_-\) are two bounded
components of \(\mathbb R^2\setminus Z\), while the remainder is connected.
Indeed, apart from \(P_+,P_-\), and \(Z\), the only part of the domain is the
negative nodal domain \(D_-\). Since \(Z\) lies at a positive distance from
the boundary, \(D_-\) is joined, through a zero-free inner neighborhood of
the boundary, to \(\mathbb R^2\setminus\overline{\Omega_\varepsilon}\) and to
the regions inside all the holes, forming a single component of
\(\mathbb R^2\setminus Z\). Therefore,
\begin{equation}
  F=3.
  \tag{6.5}
\end{equation}

Euler's formula for a finite planar graph is
\begin{equation}
  V-E+F=1+c.
  \tag{6.6}
\end{equation}
Every vertex has degree at least \(2\), so the handshake lemma gives
\begin{equation}
  2E=\sum_v\deg(v)\geq2V,
  \qquad E\geq V.
  \tag{6.7}
\end{equation}
Substituting \(F=3\) into (6.6), we obtain
\(3=1+c+E-V\geq1+c\), and hence \(c\leq2\). Together with \(c\geq2\), this
shows that \(c=2\) and \(E=V\). Equality must therefore hold in (6.7), and
every vertex has degree \(2\); in particular, there is no singular zero of
degree at least \(4\). Each connected component is consequently a compact,
connected, regular real-analytic one-dimensional manifold without boundary,
and hence a real-analytic simple closed curve. The two components lie in
\(Q_+\) and \(Q_-\), respectively. This proves the theorem.
\end{proof}

\begin{remark}\label{rem:reflection-symmetry}
The reflection symmetry of the spectrum is used only to construct and
identify the fixed polygonal counterexample. The smooth approximation in the
next section need not preserve reflection symmetry. As long as the two-sided
spectral gap about the third eigenvalue and the two groups of strict sign
inequalities are preserved, the same Courant count recovers all the
conclusions.
\end{remark}

\section{Smoothing and Proof of the Main Theorem}

Fix one of the domains \(\Omega_\varepsilon\) from
Theorem~\ref{thm:polygonal-counterexample}, abbreviate it as \(\Omega_0\), and
write \(U_0\) for the corresponding third eigenfunction. Set
\begin{equation}
  K=\overline{Q_+}\cup\overline{Q_-}\Subset\Omega_0.
  \tag{7.1}
\end{equation}

\begin{lemma}[Connected smooth inner exhaustion]\label{lem:smooth-exhaustion}
There exist bounded connected \(C^\infty\) domains \(\Omega_m\) such that
\begin{equation}
  K\Subset\Omega_m\Subset\Omega_{m+1}\Subset\Omega_0,
  \qquad
  \bigcup_{m=1}^{\infty}\Omega_m=\Omega_0.
  \tag{7.2}
\end{equation}
For every fixed \(j\),
\begin{equation}
  \lambda_j(\Omega_m)\longrightarrow\lambda_j(\Omega_0).
  \tag{7.3}
\end{equation}
\end{lemma}

\begin{proof}
The boundary of the polygonal domain \(\Omega_0\) has only finitely many
edges and finitely many corners, and the thin neck has fixed positive width.
Choose a smooth regularized distance function \(d_*\) that is comparable
from above and below in \(\Omega_0\) with
\(\operatorname{dist}(x,\partial\Omega_0)\). Choose regular values
\(t_m\downarrow0\), and let \(\Omega_m\) be the connected component of
\(\{d_*>t_m\}\) that contains the fixed compact set \(K\) and the centerline
of the thin neck. For all sufficiently small \(t_m\), only a thin boundary
layer is removed; the set remains connected through the fixed-width thin
neck and retains every hole. By discarding finitely many initial indices and
choosing the decreasing regular values appropriately, the sets can be made
nested. The regular value theorem gives
\(\partial\Omega_m\in C^\infty\), while the distance comparison gives the
exhaustion relations (7.2).

By domain monotonicity, \(\lambda_j(\Omega_m)\) decreases monotonically with
\(m\) and is bounded below by \(\lambda_j(\Omega_0)\). It therefore has a
limit \(\ell_j\geq\lambda_j(\Omega_0)\). To prove the reverse inequality,
fix \(\tau>0\). By the min--max formula, one can choose a \(j\)-dimensional
subspace \(E\subset H_0^1(\Omega_0)\) whose maximal Rayleigh quotient is less
than \(\lambda_j(\Omega_0)+\tau/2\). Using the density of
\(C_c^\infty(\Omega_0)\) in \(H_0^1(\Omega_0)\), approximate the members of a
basis of \(E\) one by one. If the approximation is sufficiently close, the
resulting \(j\)-dimensional smooth compactly supported subspace has maximal
Rayleigh quotient less than \(\lambda_j(\Omega_0)+\tau\). The union of the
finitely many supports is a compact subset of \(\Omega_0\), and is therefore
eventually contained in some \(\Omega_m\). Testing with this subspace in the
min--max formula for \(\Omega_m\) yields
\[
  \lambda_j(\Omega_m)\leq\lambda_j(\Omega_0)+\tau.
\]
Letting \(m\to\infty\) and then \(\tau\downarrow0\) proves (7.3).
\end{proof}

\begin{lemma}[Local convergence of a simple eigenfunction under inner
exhaustion]\label{lem:eigenfunction-convergence}
Let \(U_m\) be an \(L^2\)-normalized real eigenfunction corresponding to the
third eigenvalue on \(\Omega_m\), and extend it by zero to \(\Omega_0\). If
its sign is fixed by
\(\int_{\Omega_m}U_mU_0>0\), then
\begin{equation}
  U_m\longrightarrow U_0 \quad\text{in }L^2(\Omega_0),
  \qquad
  U_m\longrightarrow U_0 \quad\text{in }C^1(K).
  \tag{7.4}
\end{equation}
\end{lemma}

\begin{proof}
By Lemma~\ref{lem:smooth-exhaustion}, \(\lambda_3(\Omega_m)\) is bounded. The
zero extensions of the \(U_m\) belong to \(H_0^1(\Omega_0)\) and satisfy
\[
  \lVert U_m\rVert_{L^2(\Omega_0)}=1,
  \qquad
  \lVert\nabla U_m\rVert_{L^2(\Omega_0)}^2=\lambda_3(\Omega_m).
\]
Thus every subsequence has a further subsequence converging weakly in
\(H_0^1(\Omega_0)\) and strongly in \(L^2(\Omega_0)\) to some normalized
function \(W\). For any \(\zeta\in C_c^\infty(\Omega_0)\), its support is
eventually contained in \(\Omega_m\), so \(\zeta\) may be used as a test
function in the weak eigenvalue equation for \(U_m\). Letting
\(m\to\infty\) and using (7.3), we obtain
\[
  -\Delta W=\lambda_3(\Omega_0)W.
\]
This eigenvalue is simple by Proposition~5.2, so \(W=U_0\) or
\(W=-U_0\); the choice of orientation excludes the minus sign. Since all
subsequential limits coincide, the full sequence converges in \(L^2\).

For all sufficiently large \(m\), \(K\Subset\Omega_m\). On a fixed slightly
larger neighborhood of \(K\), the functions \(U_m\) and \(U_0\) satisfy
Helmholtz equations with the same coefficients and with convergent
eigenvalues. The same interior elliptic estimates and Sobolev embedding as
in the proof of Proposition~4.3 upgrade the \(L^2\)-convergence to
\(C^1(K)\)-convergence.
\end{proof}

\begin{proof}[Proof of Theorem~1.2]
Write the two-sided spectral gap of the fixed polygonal counterexample as
\begin{equation}
  g=\min\bigl\{
    \lambda_3(\Omega_0)-\lambda_2(\Omega_0),
    \lambda_4(\Omega_0)-\lambda_3(\Omega_0)
  \bigr\}>0.
  \tag{7.5}
\end{equation}
By (7.3), for all sufficiently large \(m\),
\begin{equation}
  \lambda_3(\Omega_m)-\lambda_2(\Omega_m)>\frac{g}{2},
  \qquad
  \lambda_4(\Omega_m)-\lambda_3(\Omega_m)>\frac{g}{2}.
  \tag{7.6}
\end{equation}
Thus the third eigenvalue of \(\Omega_m\) is simple.

On the other hand, Lemma~\ref{lem:eigenfunction-convergence} and (6.1) show
that, for the same sufficiently large \(m\),
\begin{equation}
  U_m(p_\pm)>0,
  \qquad
  \max_{\Gamma_\pm}U_m<0.
  \tag{7.7}
\end{equation}
We may now repeat verbatim the nodal-domain argument from
Theorem~\ref{thm:polygonal-counterexample}: two positive nodal domains are
confined separately within \(Q_+\) and \(Q_-\), the negative set is nonempty,
and Courant's theorem bounds the total number of nodal domains by three. The
interior-zero sign-change lemma therefore gives
\[
  \mathcal Z(U_m)
    \subset\overline{P_{m,+}}\cup\overline{P_{m,-}}
    \Subset Q_+\cup Q_-\Subset\Omega_m.
\]
Applying the same finite analytic planar-graph argument and Euler count to
this compact interior nodal set shows that it has exactly two connected
components and that neither component has a singular vertex. Consequently,
both are real-analytic simple closed curves. Choose one such sufficiently
large \(m\) and set \(\Omega=\Omega_m\) and \(U=U_m\). This yields (1.4),
(1.5), and the conclusion concerning the two closed curves simultaneously.
\end{proof}

\end{document}